\documentclass[12pt]{amsart}

\usepackage{amsfonts,amsthm,amsmath,amssymb,latexsym}
\usepackage{graphicx,color}
\usepackage[all]{xy}

\usepackage{amssymb}
\usepackage{soul}
\usepackage{epstopdf}

\usepackage{geometry}                % See geometry.pdf to learn the layout options. There are lots.
\usepackage{graphicx}
\usepackage{amsfonts,amsthm,amsmath,amssymb,latexsym, amscd, euscript}
\usepackage{graphicx,color}
\usepackage[all]{xy}
\usepackage{epsfig}
\usepackage{mathrsfs}

\usepackage{amssymb}
\usepackage{epstopdf}
\begin{document}

\theoremstyle{definition} %%% for statements in roman typeface

 \newtheorem{definition}{Definition}[section]
 \newtheorem{remark}[definition]{Remark}
 \newtheorem{example}[definition]{Example}

\newtheorem*{notation}{Notation}  %%% for statements without numbering

\theoremstyle{plain}      %%% for statements in italic typeface

 \newtheorem{proposition}[definition]{Proposition}
 \newtheorem{theorem}[definition]{Theorem}
 \newtheorem{corollary}[definition]{Corollary}
 \newtheorem{lemma}[definition]{Lemma}

\def\H{{\mathbb H}}
\def\F{{\mathcal F}}
\def\R{{\mathbb R}}
\def\Q{\hat{\mathbb Q}}
\def\Z{{\mathbb Z}}
\def\E{{\mathcal E}}
\def\N{{\mathbb N}}
\def\X{{\mathcal X}}
\def\Y{{\mathcal Y}}
\def\C{{\mathbb C}}
\def\D{{\mathbb D}}
\def\G{{\mathcal G}}
\def\T{{\mathcal T}}

\title[Intersections between foliations]{Intersection numbers between horizontal foliations of quadratic differentials}


\date{\today}

\author{Dragomir \v Sari\' c}
\address{DS: Mathematics PhD. Program, Graduate Center of the City University of New York, 365 Fifth Avenue, New York, NY 10016-4309}
\address{DS: Department of Mathematics, Queens College of the City University of New York, 65--30 Kissena Blvd., Flushing, NY 11367}
\email{Dragomir.Saric@qc.cuny.edu}

\author{Taro Shima}
\address{TS:  Mathematics PhD. Program, Graduate Center of the City University of New York, 365 Fifth Avenue, New York, NY 10016-4309}
\email{tshima@gradcenter.cuny.edu}
\maketitle

\begin{abstract}
We establish that the intersection number between the horizontal foliations of any two finite-area holomorphic quadratic differentials on an arbitrary Riemann surface is finite. Our main result shows that the intersection number is jointly continuous in the $L^1$-norm on the quadratic differentials. A corollary is that the Jenkins-Strebel differentials are not dense in the space of all finite-area holomorphic quadratic differentials when the infinite Riemann surface is not parabolic.
\end{abstract}

\section{Introduction}

Let $X=\mathbb{H}/\Gamma$, $\Gamma <PSL_2(\mathbb{R})$, be an arbitrary Riemann surface equipped with the hyperbolic metric.  Denote by $A(X)$ the space of all finite-area holomorphic quadratic differentials $\varphi$  on $X$ equipped with the $L^1$-norm
$$
\|\varphi\|=\int_X|\varphi |.
$$
Each non-zero element $\varphi\in A(X)$ induces a (horizontal) measured foliation on $X$ by the differential $Im (\sqrt{\varphi (z)}dz)$. It turns out that each (regular) leaf of the horizontal foliation $\mathcal{F}_{\varphi}$ of $\varphi$ is homotopic to a simple geodesic on $X$ (\cite{MardenStrebel1}). We push forward the transverse measure to obtain a measured geodesic lamination $\mu_{\varphi}$ on $X$. Marden and Strebel \cite{MardenStrebel} proved that the map $\varphi\mapsto\mu_{\varphi}$ is injective when $X$ is a parabolic surface, and this was extended to arbitrary surfaces in \cite{Saric22}. When $X=\mathbb{H}/\Gamma$ is a compact surface, Hubbard and Masur \cite{HMmain} proved that every measured lamination on $X$ is homotopic (i.e., realized in the homotopy class) to the horizontal foliation of a unique holomorphic quadratic differential. When $X$ is not of finite (hyperbolic) area, many measured laminations cannot be realized by finite-area holomorphic quadratic differentials (see \cite{Saric22}). In \cite{Saric24} and \cite[Theorem 1.1]{Saric25}, it was established that on an arbitrary $X=\mathbb{H}/\Gamma$, a measured lamination can be realized by a finite-area holomorphic quadratic differential if and only if it is homotopic to a partial measured foliation with finite Dirichlet integral.

The space of finite-area holomorphic differentials $A(X)$ plays the role of the cotangent space to the Teichm\"uller space $T(X)$ at the (marked) point $X$ (see \cite{Ahlfors}, \cite{GardinerLakic}). In addition, the space $A(X)$ is associated to the Teichm\"uller extremal maps, Strebel points, and extremal and uniquely extremal quasiconformal maps (see \cite{ReichStrebel, BLMM, GardinerLakic}). The space $A(X)$ also plays a role in studying the harmonic mappings (see \cite{Sagman}).

In this paper, we continue the study of the finite-area holomorphic quadratic differentials from the point of view of their horizontal foliations and measured laminations on infinite Riemann surfaces. Given $\varphi\in A(X)$, the corresponding measured lamination $\mu_{\varphi}$, in general, has an infinite length (the length is locally the hyperbolic length of leaves times the transverse measure) on $X$. However, the intersection number between $\mu_{\varphi}$ and $\mu_{\psi}$ is finite for all $\varphi ,\psi\in A(X)$. This follows from the following formula (see Proposition \ref{thm:inters-extremal_l}), which generalizes a result of Minsky \cite{Minsky} on compact surfaces.

\begin{proposition}
\label{thm:intersection}
Let $X=\mathbb{H}/\Gamma$ for any Fuchsian group $\Gamma <PSL_2(\mathbb{R})$. For any non-zero $\varphi ,\psi\in A(X)$, we have
$$
i(\mu_{\varphi},\mu_{\psi})^2\leq(\int_X|\varphi | )\cdot (\int_X|\psi |),
$$
where $i(\cdot ,\cdot )$ is the intersection number between two measured laminations. In particular, $i(\mu_{\varphi},\mu_{\psi})<\infty$. 
\end{proposition}

The main result of the paper is that the intersection number is continuous for the $L^1$-norm on the space $A(X)$ (see Theorem \ref{thm:int_numb_cont}). This requires an analysis of the trajectories of finite-area holomorphic quadratic differentials, which is necessarily divided into three cases: $X=\mathbb{H}$ (the universal case), $X=\mathbb{H}/\Gamma$ with $\Gamma$ of the first kind and $\Gamma$ of the second kind.

\begin{theorem}
\label{thm:cont-intersection}
The intersection number defines a function
$$
i:A(X)\times A(X)\to\mathbb{R}.
$$
The function $i$ is jointly continuous for the $L^1$-norm.
\end{theorem}

A finite-area holomorphic quadratic differential is called Jenkins-Strebel if, outside of a set of zero area, all horizontal trajectories are closed and homotopic and necessarily form a single cylinder (for the existence, see Strebel \cite[Chapter VI]{Strebel}). In \cite{Saric24}, it was proved that the Teichm\"uller distance on $T(X)$ is given by the supremum of the logarithm of the ratios of the extremal lengths over all simple closed curves on a {\it parabolic} Riemann surface  $X$. This result depends on the fact that the Jenkins-Strebel differentials are dense in $A(X)$ for $X$ a parabolic Riemann surface. Using the continuity of the intersection numbers, we show that the Jenkins-Strebel differentials are not dense on an infinite Riemann surface with an infinitely generated fundamental group that is not parabolic.

\begin{theorem}
\label{thm:non-density}
There exists a non-parabolic Riemann surface with an infinitely generated fundamental group and a finite-area holomorphic quadratic differential $\varphi$ that cannot be approximated by a sequence of Jenkins-Strebel differentials in the $L^1$-norm. Almost every horizontal trajectory of $\varphi$ is transient.
\end{theorem}

The above Theorem provides a clue to how Teichm\"uller spaces of parabolic and non-parabolic Riemann surfaces may differ. In \S 3, we give a new definition of the extremal length of measured foliations which works for all Riemann surfaces.

\section{Integrable, proper partial measured foliations}

Recall the definition of an integrable, proper partial measured foliation on $X$ from \cite[Definition 2.1]{Saric24}.
\begin{definition}
\label{def:partial-fol}
Let $X=\mathbb{H}/\Gamma$ be an arbitrary Riemann surface.
A {\it partial measured foliation} $\mathcal{F}$ on $X$ consists of countably many triples $\{ (v_i, U_i,E_i)\}_i$ where $U_i\subset X$ is a closed Jordan domain, $E_i\subset U_i$ is a measurable set and $v_i:U_i\to\mathbb{R}$ is a continuous function satisfying the following conditions:
\begin{enumerate}
\item[(i)] The family $\{ U_i\}_i$ is locally finite on $X$.

\item[(ii)]
The boundary of $U_i$, denoted as $\partial U_i$, is piecewise smooth and is divided into four closed arcs with non-overlapping interiors. Within this partition, two specific opposite arcs, namely $a_i^1$ and $a_i^2$, are chosen as the vertical sides of $\partial U_i$.

\item[(iii)]
The functions $v_i:U_i\to\mathbb{R}$ are differentiable with surjective $dv_i$ at the tangent space of each point of the interior of $U_i$. 

By the Implicit Function Theorem, the pre-image $v_i^{-1}(c)$ for $c\in\mathbb{R}$, if non-empty, is an open differentiable arc. We require that $v_i^{-1}(c)$ at one end accumulates to a unique point in $a_i^1\subset \partial U_i$ and the other end to a unique point in $a_i^2\subset \partial U_i$ and that $E_i$ is foliated by the differentiable arcs  $v_i^{-1}(c)$ for $c\in\mathbb{R}$.

\item[(iv)] For any two sets $U_i$ and $U_j$,  we have
 \begin{equation}
\label{eq:coord-change}
v_i=\pm v_j+const
\end{equation} 
on $U_i\cap U_j$.

\item[(v)] For any two $E_i$ and $E_j$, we have
$$
E_i\cap \bar{U}_j=E_j\cap\bar{U}_i
$$
where $\bar{U}$ is the closure of $U$. 
\end{enumerate}
\end{definition}

Since $dv_i$ are surjective, the Implicit Function Theorem implies that the sets $U_i$ are foliated by differentiable arcs $v_i^{-1}(c)$ for $c\in\mathbb{R}$. Therefore, $U_i$ has a product structure with leaves being horizontal arcs. We only consider the foliation arcs in $E_i$ to comprise the horizontal leaves of the partial foliation. A {\it horizontal trajectory} is a maximal extension of a horizontal leaf of a single chart.

A partial measured foliation is {\it proper} if each component of a lift to $\mathbb{H}$ of each horizontal trajectory of $\mathcal{F}$ on $X$ has two endpoints on $\partial_{\infty}\mathbb{H}$. The Dirichlet integral $D(v_i)=\int_{U_i}|\nabla v_i|$ is well-defined locally, and by taking a partition of unity, we define the Dirichlet integral $D(\mathcal{F})$ of $\mathcal{F}$ on $X$; it can be finite or infinite. If $D(\mathcal{F})<\infty$ then we say that $\mathcal{F}$ is integrable.  

We will say that two proper, partial measured foliations $\mathcal{F}_1$ and $\mathcal{F}_2$ are {\it equivalent} if there is a homotopy between lifts $\widetilde{\mathcal{F}}_1$ and $\widetilde{\mathcal{F}}_2$ to the universal cover $\mathbb{H}$ such that the push-forward (under the homotopy) of the transverse measure on $\widetilde{\mathcal{F}}_1$ equals the the transverse measure on $\widetilde{\mathcal{F}}_2$.

If $\varphi$ is a finite-area holomorphic quadratic differential on $X$, then the horizontal (measured) foliation $\mathcal{F}_{\varphi}$ of $\varphi$ given by $Im (\sqrt{\varphi (z)}dz)$ is a proper integrable partial foliation as above. The converse is true
(see \cite[Theorem 1.6]{Saric24} and \cite[Theorem 1.1]{Saric25})

\begin{theorem}
\label{thm:mf=qd}
Let $X=\mathbb{H}/\Gamma$ be an arbitrary Riemann surface. Given any integrable, proper partial measured foliation $\mathcal{F}$ on $X$, there exists a unique finite-area holomorphic quadratic differential $\varphi$ on $X$ whose horizontal foliation is equivalent to $\mathcal{F}$.
\end{theorem}

\section{The Dirichlet principle and extremal metrics}
In this section, we give a new definition of an extremal length of a measured foliation which works for all Riemann surfaces and prove that the horizontal foliation of a holomorphic quadratic differential has a minimal Dirichlet integral over all partial measured laminations homotopic to the horizontal foliation (Theorem \ref{thm:dir_principle}). 

\subsection{A new definition of the extremal length of measured foliations}
In the case of a compact Riemann surface, the extremal length of a measured foliation $\mathcal{F}$ is defined to be $\int_X|\varphi_{\mathcal{F}}|$, where $\varphi_{\mathcal{F}}$ is the unique holomorphic quadratic differential whose horizontal foliation is homotopic to $\mathcal{F}$ (see Kerckhoff \cite{Ker}). This formula coincides with the usual definition of the extremal length of the curve family on $X$ homotopic to a simple closed geodesic, and the general form is justified by the fact that every holomorphic quadratic differential on a compact Riemann surface is the limit (in the $L^1$-norm) of a sequence of Jenkins-Strebel differentials. 

When $X$ is an infinite Riemann surface that is not parabolic, one does not expect to have this approximation result. In fact, on $X=\mathbb{H}$, there are no Jenkins-Strebel differentials. We introduce an alternative definition of the extremal length of an integrable measured foliation that works for all surfaces without the use of holomorphic quadratic differentials.

\begin{definition}
\label{def:ext_l_measured_g_l}
Let $\mu$ be a measured geodesic lamination on an arbitrary Riemann surface $X=\mathbb{H}/\Gamma$. Assume that there exists a proper integrable partial measured foliation $\mathcal{F}$ that realizes $\mu$, i.e., $\mu$ and $\mathcal{F}$ have the same intersection numbers with the homotopy classes of simple closed curves (cross-cuts if $\Gamma$ is of the second kind including $\Gamma =\{ id\}$, i.e., $X=\mathbb{H}$). The {\it extremal length} $EL(\mu )$ of $\mu$ is defined by
$$
EL(\mu )=\inf_{\mathcal{F}} D(\mathcal{F}),
$$
the infimum is taken over all $\mathcal{F}$ that realize $\mu$. 
\end{definition}

We show below that the standard definition of the extremal length of a weighted simple closed multi-curve agrees with our definition. First, we introduce some definitions.

A {\it geodesic current} $\alpha$ on $X$ is a measure $\tilde{\alpha}$ on the space of geodesics $G(\mathbb{H})=(\mathbb{H}\times\mathbb{H}\setminus{\mathrm{diag}})/\mathbb{Z}_2$ that is invariant under the action of the covering group $\Gamma$ (see \cite{Bonahon}, \cite{BonS}). The {\it support} of $\alpha$ is the set of geodesics on $X$ that form the image under the projection of the covering map $\mathbb{H}\to X$ of the support of $\tilde{\alpha}$ on $G(\mathbb{H})$.  

A {\it measured lamination} is a geodesic current $\alpha$ whose support is a geodesic lamination, i.e., no two geodesics of the support intersect (see \cite{Thurston}). The support of the lift $\tilde{\alpha}$ of a measured lamination $\alpha$ on $X$ is a $\Gamma$-invariant measure on $G(\mathbb{H})$ whose support is a geodesic lamination on $\mathbb{H}$. The geodesics of the support of $\tilde{\alpha}$ are pairwise mutually disjoint, but they can have ideal endpoints in common. 

Let $C$ be a simple closed curve on $X$ that is not homotopic to a point or a puncture. Given a conformal metric $\rho$ on $X$, define the length of a rectifiable curve $C'$ by $\ell_{\rho}(C')=\int_{C'}\rho (z)|dz|$, where the infimum is over all simple closed curves $C'$ homotopic to $C$. Let $A_{\rho}(X)=\int_X \rho^2(z)dxdy$. The {\it extremal length} $EL(C)$ of the (homotopy class of) curve $C$ is defined by
$$
EL(C)=\sup_{\rho} \frac{\ell_{\rho}(C)^2}{A_{\rho}(X)},
$$
where the supremum is taken over all conformal metrics $\rho$ with $0<A_{\rho}(X)<\infty$. 
A metric $\rho_0$ is {\it extremal} for $C$ is $EL(C)=\frac{\ell_{\rho_0}(C)^2}{A_{\rho_0}(X)}$.

Hubbard and Masur \cite{HubMas}, and Renelt \cite{Renelt} proved that there exists a unique holomorphic quadratic differential $\varphi_{C,1}$ with a single cylinder of non-singular closed horizontal trajectories homotopic to $C$ of height $1$ with the complement of the cylinder having zero area (see Strebel \cite[Theorem 21.1]{Strebel}). The metric $\rho_0=|\varphi_{C,1}|^\frac{1}{2}$ is extremal for $C$.  Let $C'$ be any closed horizontal trajectory of $\varphi_{C,1}$ and $\ell =\int_{C'}|\varphi_{C,1}(z)|^{\frac{1}{2}}|dz|$. Then we have
$$
EL(C)=\frac{[\int_{C'}|\varphi_{C,1}(z)|^{\frac{1}{2}}|dz|]^2}{\int_X|\varphi_{C,1}(z)|dxdy}=\frac{\ell^2}{1\cdot\ell}=\ell =\int_X|\varphi_{C,1}(z)|dxdy,
$$
because the height is $1$.

We normalize the above formula slightly differently, which lends to extending the extremal length to weighted multicurves and even to measured laminations of $X$. Let $\varphi_{C,b}=b^2\varphi_{C,1}$ be the holomorphic quadratic differential with the same cylinder as $\varphi_{C,1}$ but with height $b$. 

We can, therefore, consider this cylinder as a foliation that realizes the weighted curve $(C,b)$ and define the {\it extremal length} of $(C,b)$ to be
$$
EL(C,b):=b^2EL(C)=\int_X|\varphi_{C,b}(z)|dxdy.
$$

For any other (partial) foliation $\mathcal{F}$ of $X$ which realizes $(C,b)$, we have that 
$$
ht_{\mathcal{F}}(B)=ht_{\varphi_{C,b}}(B)
$$
for all simple closed curves $B$ in $X$. By the definition
$
ht_{\mathcal{F}}(B)=\inf_{B'\tilde B}\int_{B'}d\mathcal{F}$, where the infimum is over all $B'$ homotopic to $B$. Then, the Dirichlet principle \cite[Theorem 24.5]{Strebel} says that
$$
D(\mathcal{F})\geq \int_X|\varphi_{C,b}|.
$$

Therefore, we can define (Definition \ref{def:ext_l_measured_g_l}) the extremal length of the weighted curve $(C,b)$ by
\begin{equation}
\label{eq:ext_l_one_curve}
EL(C,b)=\inf_{\mathcal{F}} D(\mathcal{F}),
\end{equation}
infimum is taken over all partial foliations that realize $(C,b)$.

The formula (\ref{eq:ext_l_one_curve}) holds even for multi-curves  $\{ (C_n,b_n)\}$, where $C_n$ are simple closed curves that are pairwise disjoint, non-homotopic, and not homotopic to a point or a cusp. The number of curves is either finite or countably infinite since $X$ is an infinite surface. The weights $b_n$ are positive, and there are either finitely or infinitely many curves in a multi-curve. We consider embedded annuli $R_n$ in $X$ pairwise disjoint and homotopic to $C_n$ for each $n$. Denote by $EL(R_n)$ the extremal length of curves in $R_n$ homotopic to $C_n$. Define
\begin{equation}
\label{eq:ext_l_multi_curve}
EL(\{ (C_n,b_n)\})=\inf_{\{R_n\}}\sum_nb_n^2 \cdot EL(R_n),
\end{equation}
infimum is taken over all families of annuli $\{ R_n\}$.

If the sum in (\ref{eq:ext_l_multi_curve}) is finite for some choice of annuli $\{ R_n\}$, then (by \cite[Theorem 22.1]{Strebel}) there exists a finite-area holomorphic quadratic differential $\varphi$ that realizes the multi-curve $\{ (C_n,b_n)\}$. In particular, the horizontal foliation of $\varphi$ has a cylinder of height $b_n$ homotopic to $C_n$ for each $n$, and the complement of the union of the cylinders has zero area. In addition, we have
$$
EL(\{ (C_n,b_n)\})=\int_X|\varphi (z)|dxdy. 
$$

If $\mathcal{F}_{\{ (C_nb_n)\}}$ is an arbitrary foliation that realizes the multi-curve $\{ (C_n,b_n)\}$, then (by \cite[Theorem 22.1]{Strebel})
$$
D(\mathcal{F}_{\{ (C_nb_n)\}})\geq \int_X|\varphi (z)|dxdy.
$$
Therefore
$$
EL(\{ (C_n,b_n)\})=\inf_{\mathcal{F}}D(\mathcal{F}),
$$
the infimum is taken over all partial foliations $\mathcal{F}$ that realize the multi-curve $\{ (C_n,b_n)\}$.

\subsection{The Dirichlet principle}
 Given the homotopy class of an integrable, proper, partial measured foliation, we prove that the horizontal foliation of the unique holomorphic quadratic differential in the homotopy class minimizes the Dirichlet integral.

\begin{theorem}[The Dirichlet principle]
\label{thm:dir_principle}
Let $X=\mathbb{H}/\Gamma$ be an arbitrary Riemann surface and $\varphi$ be a finite-area holomorphic quadratic differential on $X$. For any proper measured lamination $\mathcal{F}$ on $X$ with $D(\mathcal{F})<\infty$ that is equivalent to the horizontal foliation of $\varphi$, we have
$$
\int_X|\varphi (z)|dxdy\leq D(\mathcal{F}). 
$$
\end{theorem}

\begin{proof}
This is proved for parabolic Riemann surfaces by Marden and Strebel \cite{MardenStrebel} (see also Strebel \cite[Theorem 24.2]{Strebel}). The argument uses the length-area method and a decomposition of the vertical foliation of $\varphi\in A(X)$ into cylindrical domains and spiral domains. The inequality is achieved for each domain separately. 

To prove the theorem for an arbitrary Riemann surface, we must consider the cross-cut strips (since the complement of the three types of domains in $X$ has zero area, see Strebel \cite[\S 13]{Strebel}). The argument is similar to the argument for the cylinders, and the theorem holds for arbitrary Riemann surfaces (see also \cite{Saric25}).
\end{proof}

Finally, there is an extremal metric for $EL(\mathcal{F})$ given by $$\rho_0(z)=\sqrt{|\varphi_{\mathcal{F}}(z)|}|dz|,$$ where $\varphi_{\mathcal{F}}$ is the finite area holomorphic quadratic differential whose horizontal foliation is homotopic to $\mathcal{F}$. By Theorem \ref{thm:mf=qd}, the holomorphic quadratic differential $\varphi_{\mathcal{F}}$ exists. 

\section{The convergence in $A(X)$}
The space of integrable holomorphic quadratic differentials is equipped with the topology induced by the $L^1$-norms. It is helpful to describe this topology in terms of sequential convergence.

\begin{lemma} 
\label{lem:L^1-conv}
A sequence $\varphi_n\in A(X)$ converges to $\varphi\in A(X)$ in the $L^1$-norm if and only if 
\begin{enumerate}
\item
$\varphi_n(z)$ converges to $\varphi (z)$ uniformly on compact subsets of $X$, and 
\item $\limsup_{n\to\infty}\int_X|\varphi_n|\leq \int_X|\varphi |$.%  as $n\to\infty$. 
\end{enumerate}
\end{lemma}
\begin{proof}
To prove the only if direction, note that the inequality $0\leq |\varphi -\varphi_n|+|\varphi |-|\varphi_n|\leq 2|\varphi |$ allows us to apply Lebesgue's Dominated Convergence Theorem to $|\varphi -\varphi_n|+|\varphi |-|\varphi_n|$ which gives $0\leq \lim_{n\to\infty}\int_X|\varphi -\varphi_n|=\lim_{n\to\infty}\int_X|\varphi_n|-\int_X|\varphi |\leq 0$. For the if direction, the uniform convergence on compact sets follows from the Cauchy Theorem for holomorphic function and the inequality $||\varphi |-|\varphi_n||\leq |\varphi -\varphi_n|$ implies $\int_X|\varphi_n|\to\int_X|\varphi |$ as $n\to\infty$. 
\end{proof}

Using the above theorem, we give an example of a sequence of finite area holomorphic quadratic differentials with uniformly bounded $L^1$-norms that converge locally uniformly to a finite area holomorphic quadratic differential, but they do not converge in the $L^1$-norm.

\begin{figure}[h]
\begin{center}
{\centerline{\epsfig{file =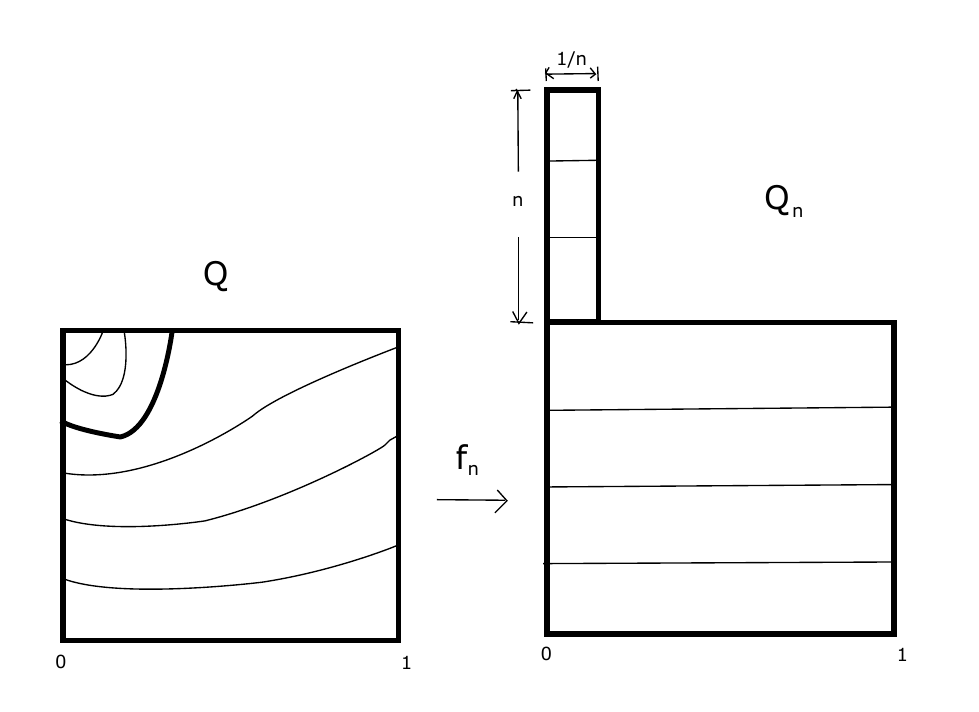,width=9.0cm,angle=0} }}
\vspace{-20pt}
\end{center}
\caption{The locally uniformly convergent sequence with bounded $L^1$-norms.}
\end{figure}

\begin{example}
Let $Q_n=[0,1]\times [0,1]\cup [0,\frac{1}{n}]\times [1,n]$ and $\psi_n(z)dz^2=dz^2$. Then $\|\psi_n\|:=\iint_{Q_n}dxdy=2$ for all $n$. Let $f_n$ be the conformal map from $Q=[0,1]\times [0,1]$ to $Q_n$ that fixes points $(0,0)$, $(1,0)$ and $(1,1)$.  Since $\cap_{n\in\mathbb{N}} Q_n=Q$, the Carath\'eodory kernel theorem implies that $f_n\to id$ uniformly on compact subsets of $Q$. Let $\varphi_n(z)dz^2=(f_n)^*\psi_n(z)dz^2=(f_n')^2(z)dz^2$. Since $f_n(z)\to z$ as $n\to\infty$, we conclude that $\varphi_n(z)\to 1$ uniformly on compact sets (see Figure 1). However, $\|\varphi_n\|=\|\psi_n\|=2$ does not converge to $\| dz^2\|=1$. By Lemma \ref{lem:L^1-conv}, we have that $\varphi_n$ does not converge in the $L^1$-norm to $dz^2$ even though it converges uniformly on compact subsets and the norms are uniformly bounded.

If we attach $[0,\frac{1}{n}]\times [1,\sqrt{n}]$ to $Q$, then the above procedure provides a sequence of differentials that converge in the $L^1$-norm to $dz^2$.
\end{example}

\section{The intersection numbers for measured laminations on an arbitrary Riemann surface}

Let $\mu$ and $\nu$ be two measured (geodesic) laminations on an arbitrary Riemann surface $X=\mathbb{H}/\Gamma$, where $\Gamma <PSL_2(\mathbb{R})$. 
Let $|\mu |$ and $|\nu |$ be the geodesic laminations supports of $\mu$ and $\nu$. The set $T$ of all transverse intersections between $|\mu |$ and $|\nu |$ may not be closed on a surface with infinite hyperbolic area. We consider the closure $\bar{T}$ of the set $T$ of transverse intersections. A point in $\bar{T}\setminus T$ lies on a common geodesic of $|\mu |$ and $|\nu |$.

Let $P\in \bar{T}$ be an arbitrary point. It is elementary to find an open neighborhood $U_P$ of $P$ such that the set of components of $|\mu |\cap U_P$ is nested in the sense that each component separates the set into two parts, except the leftmost and the rightmost component. We assume that the components of $|\nu |\cap U_P$ have the same property by choosing $U_P$ even smaller, if necessary. If $P\in T$, since $P\in U_P$ is in the intersection of two components, by decreasing $U_P$ even further, we can arrange that each pair of components of $|\mu |\cap U_P$ and $|\nu |\cap U_P$ intersects in one point. If $P\in\bar{T}\setminus T$, a pair of components of $|\mu |\cap U_P$ and $|\nu |\cap U_P$ can coincide, be disjoint, or transversely intersect in one point. The leftmost or the rightmost component of $U_p\cap |\mu |$ can be an atom of the transverse measure, and the same holds for atoms of $\nu$ (see Figure 2).

\begin{figure}[h]
\begin{center}
{\centerline{\epsfig{file =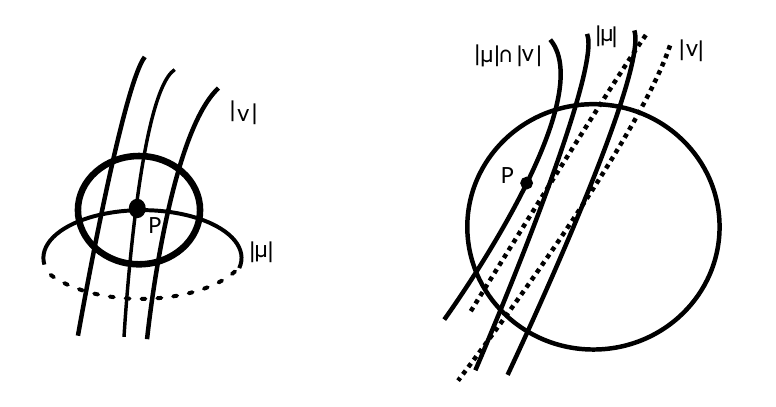,width=9.0cm,angle=0} }}
\vspace{-20pt}
\end{center}
\caption{On the left side, $|\mu |$ has a closed leaf, i.e., $\mu$ has an atom at $P$. On the right side, $|\mu |\cap |\nu |$ has a common leaf through $P$.}
\end{figure}

We denote by $R_i$ the subregion of $U_{P_i}$ that lies between the leftmost and the rightmost components of $|\mu |\cap U_{P_i}$, and the leftmost and the rightmost components of $|\nu |\cap U_{P_i}$. The subregion $R_i$ can be closed, meaning it contains all its boundary components, or it can be open, or it can contain parts of its boundary. The regions $R_i$ are called quadrilaterals, two boundary sides $|\mu |\cap R_i$ are called {\it horizontal}, and the other two sides $|\nu |\cap R_i$ are called {\it vertical}. We will use the quadrilaterals $R_i$ that contain any of their vertical or horizontal sides, or even none.

We show that there exists a countable, locally finite collection of the quadrilaterals $\{ Q_i\}_i$ covering $\bar{T}$ such that $Q_i\cap Q_j=\emptyset$ for all $i\neq j$. To prove this, 
fix a point $x_0\in X$ and define $K_n=\{ x\in X:dist (x_0,x)\leq n\}$. Then $\{ K_n\}_n$ is an increasing exhaustion of $X$ by compact sets. We choose all sets $U_P$ to have diameter less than $1$ which implies that if $U_P$ intersects $K_n$ then it cannot intersect $K_{m}\setminus K_{n+1}$ for all $m\geq n+1$. The quadrilaterals $\bar{R}_i$ constructed above from $U_{P_i}$ have the same intersection property with respect to the exhaustion $\{ K_n\}$. The choice for $\bar{R}_i$ is that it contains all of its boundary sides, and $\bar{R}_i$ contains all points of the intersections $|\mu |\cap |\nu |$ that lie in $U_{P_i}$. 

We note that if $\bar{R}_i$ and $\bar{R}_j$ intersect, then we can partition $\bar{R}_i\cup\bar{R}_j$ into finitely many quadrilaterals $\{ R_1^{i,j},\ldots ,R_k^{i,j}\}$ that are pairwise mutually disjoint. The quadrilaterals are not closed; in general, some of their boundary sides are open.

We consider the compact set $K_2\cap \bar{T}$. We cover it by the sets $U_P$ for $P\in K_2\cap \bar{T}$ and choose a finite subcover. The quadrilaterals (constructed above) inside the finite subcover by $U_P$ also cover $K_2\cap \bar{T}$. Then we make the above partition for the intersecting quadrilaterals and obtain another finite cover by pairwise disjoint quadrilaterals. Assume now that we are given a finite subcover of $K_n\cap\bar{T}$ by quadrilaterals that are pairwise disjoint. Let $cl(K_{n+1}\setminus K_{n})$ be the closure of $K_{n+1}\setminus K_{n}$.  We cover $cl(K_{n+1}\setminus K_{n})\cap\bar{T}$ by finitely many sets $U_P$ and add quadrilaterals corresponding to these sets to the covering of $K_n\cap \bar{T}$. The finite cover of $K_{n+1}\cap\bar{T}$ is not pairwise disjoint. However, the part of the cover which intersects $K_{n-1}\cap\bar{T}$ is disjoint from the added quadrilaterals by the choice of the diameter, and it is disjoint from the cover of $K_{n}\cap\bar{T}$ by the assumption. The only intersections occur for sets that cover $cl(K_{n+1}\setminus K_{n-1})\cap\bar{T}$. We again partition these sets into pairwise disjoint quadrilaterals. Thus, we obtain a disjoint cover of $\bar{T}\cap K_{n+1}$. It is important to note that when proceeding for all $m$, the part that covers $K_n\cap\bar{T}$ does not change for $m\geq n+2$. Therefore, the quadrilaterals stabilize at each compact set, and we obtain a locally finite cover of $\bar{T}$ by quadrilaterals $\{ Q_k\}$ that are pairwise disjoint. 

We define the {\it intersection number} between $\mu$ and $\nu$ by
$$
i(\mu ,\nu )=\sum_k\mu\times\nu (Q_k),
$$
where $\mu\times\nu (Q_k)$ is the product of the total transverse $\mu$-measure of the segments of $|\mu |\cap Q_k$ that transversely intersect $|\nu |\cap Q_k$, and the $\nu$-measure of the corresponding set for $\nu$.

We also note that the intersection number is independent of the choice of a local cover since any two covers with the above properties have a common refinement, and $\mu$ and $\nu$ are additive.

\section{The intersection numbers and extremal lengths}

Let $\mathcal{F}$ and $\mathcal{G}$ be two proper, integrable measured foliations on a Riemann surface $X=\mathbb{H}/\Gamma$. Let $\varphi$ and $\psi$ be the unique finite-area holomorphic quadratic differentials on $X$ whose horizontal foliations are homotopic to $\mathcal{F}$ and $\mathcal{G}$, respectively. They exist by Theorem \ref{thm:mf=qd}. We extend an inequality on compact surfaces obtained by Minsky \cite{Minsky} to arbitrary Riemann surfaces.  

\begin{proposition}
\label{thm:inters-extremal_l}
Under the above notation, we have
$$
i(\mathcal{F},\mathcal{G})^2\leq\int_X|\varphi |\int_X|\psi |.
$$
\end{proposition}

\begin{proof}
Assume that $X=\mathbb{H}/\Gamma$ with $\Gamma$ of the first kind. 
Let $X_n$ be a subsurface of $X$ of finite area obtained by cutting $X$ along simple closed curves as in \cite[\S 3]{Saric24}. Denote by $\mathcal{F}_n$ and $\mathcal{G}_n$ the partial foliation given by restricting $\mathcal{F}$ and $\mathcal{G}$ to $X_n$ and erasing leaves that can be homotoped into the boundary. Denote by $\widehat{\mathcal{F}}_n$ and $\widehat{\mathcal{G}}_n$ the double of these foliation on the double Riemann surface $\widehat{X}_n$. Let $\widehat{\varphi}_n$ and $\widehat{\psi}_n$ be the finite-area holomorphic quadratic differentials on $\widehat{X}_n$ whose horizontal foliation are homotopic to $\widehat{\mathcal{F}}_n$ and $\widehat{\mathcal{G}}_n$, respectively. The restrictions $\varphi_n=
\widehat{\varphi}_n|_{X_n}$ and $\psi_n=
\widehat{\psi}_n|_{X_n}$ converge to $\varphi$ and $\psi$ uniformly on compact subsets of $X$(see \cite[\S 3]{Saric24}). 

By a result of Minsky \cite[Lemma 5.1]{Minsky} we have the inequality
\begin{equation}
\label{eq:ineq-int-ext}
i(\widehat{\mathcal{F}}_n,\widehat{\mathcal{G}}_n)^2\leq\int_{\widehat{X}_n}|\widehat{\varphi}_n |\int_{\widehat{X}_n}|\widehat{\psi}_n |.
\end{equation}
Since $\widehat{\mathcal{F}}_n$, $\widehat{\mathcal{G}}_n$, $\widehat{\varphi}_n$, and $\widehat{\psi}_n$ are invariant under the symmetry reflection across the boundary curves of $X_n$, we have 
$i(\widehat{\mathcal{F}}_n,\widehat{\mathcal{G}}_n)=2i({\mathcal{F}}_n,{\mathcal{G}}_n)$, $\int_{\widehat{X}_n}|\widehat{\varphi}_n|=2\int_{X_n}|\varphi_n |$ and $\int_{\widehat{X}_n}|\widehat{\psi}_n|=2\int_{X_n}|\psi_n |$. Therefore we get
\begin{equation}
\label{eq:ineq-int-restri}
i({\mathcal{F}}_n,{\mathcal{G}}_n)^2\leq\int_{{X}_n}|{\varphi}_n |\int_{{X}_n}|{\psi}_n |.
\end{equation}

The measured foliations $\mathcal{F}_n$ and $\mathcal{G}_n$ are obtained by restricting $\mathcal{F}$ and $\mathcal{G}$ to $X_n$ and erasing the leaves homotopic to the boundary. Denote by $\mu_{\mathcal{F}}$ the measured geodesic lamination, which is homotopic to $\mathcal{F}$, and we use the same notation for all other induced measured lamination from foliations.

The intersection number $i(\mu_{\mathcal{F}},\mu_{\mathcal{G}})$ is defined locally by the product of the transverse measures $\mu_{\mathcal{F}}$ and $\mu_{\mathcal{G}}$ in quadrilaterals, where the leaves of $\mu_{\mathcal{F}}$ connect the vertical sides and the leaves of $\mu_{\mathcal{G}}$ connect the horizontal sides. We partition $X$ into such quadrilaterals $\{ Q_k\}_{k=1}^{\infty}$ and compute the intersection number
$$
i(\mu_{\mathcal{F}},\mu_{\mathcal{G}})=\sum_{k=1}^{\infty} \mu_{\mathcal{F}}(\partial_vQ_k)\mu_{\mathcal{G}}(\partial_hQ_k),
$$
where $\partial_vQ_k$ is one of the two vertical sides of $Q_k$ with $\mu_{\mathcal{F}}(\partial_vQ_k)$ the transverse measure, and $\partial_hQ_h$  is one of the two horizontal sides of $Q_k$ with $\mu_{\mathcal{G}}(\partial_hQ_k)$ the transverse measure. We note that the intersection number is independent of the choice of quadrilaterals. Indeed, given two locally finite partitions, we can obtain a locally finite refinement partition because an intersection of two quadilaterals is a quadrilateral, and the complement of the intersection is a finite union of quadrilaterals. 

If $i(\mu_{\mathcal{F}},\mu_{\mathcal{G}})<\infty$, then given any $\epsilon >0$ there exists $k_0=k_0(\epsilon )$ such that
$i(\mu_{\mathcal{F}},\mu_{\mathcal{G}})<\sum_{i=1}^{k_0} \mu_{\mathcal{F}}(\partial_vQ_i)\mu_{\mathcal{G}}(\partial_hQ_i)+\epsilon$. There exists $n_0=n_0(k_0)$ such that all the leaves of supports of $\mu_{\mathcal{F}}$ and $\mu_{\mathcal{G}}$ that intersect in the quadrilateral $\{ Q_k\}_{k=1}^{k_0}$ are not erased in $X_n$ for all $n\geq n_0$. Therefore, they are in supports of $\mu_{\mathcal{F}_{n_0}}$ and $\mu_{\mathcal{G}_{n_0}}$ and they contribute to the intersection $i( \mu_{\mathcal{F}_{n_0}}, \mu_{\mathcal{G}_{n_0}})$. Therefore $|i(\mu_{\mathcal{F}},\mu_{\mathcal{G}})- i( \mu_{\mathcal{F}_{n}}, \mu_{\mathcal{G}_{n}})|<\epsilon $ for all $n\geq n_0$. We conclude that $\lim_{n\to\infty} i( \mu_{\mathcal{F}_{n}}, \mu_{\mathcal{G}_{n}})=i(\mu_{\mathcal{F}},\mu_{\mathcal{G}})$. If $i(\mu_{\mathcal{F}},\mu_{\mathcal{G}})=\infty$, a similar argument gives that $\lim_{n\to\infty} i( \mu_{\mathcal{F}_{n}}, \mu_{\mathcal{G}_{n}})=\infty$. 

We consider $\limsup_{n\to\infty}\int_{X_n}|\varphi_n|$. By \cite[Lemma 3.7]{Saric24}, there exists a subsequence $\varphi_{n_k}$ of $\varphi_n$ that converges to $\varphi$ uniformly on compact subsets of $X$. By Theorem \ref{thm:dir_principle}
$$
\int_{X_n}|\varphi_n|\leq D_{X_n}(\mathcal{F}_n).
$$
Note that $D_{X_n}(\mathcal{F}_n)\leq D_X(\mathcal{F})$ and $D_X(\mathcal{F})=\int_X|\varphi |$, and similar properties hold for $\mathcal{G}$ and $\psi$. Therefore, by letting $n$ go to infinity in (\ref{eq:ineq-int-restri}), we get
\begin{equation*}
\begin{split}
i(\mu_{\varphi},\mu_{\psi})^2=\limsup_{n\to\infty} i(\mu_{\varphi_n},\mu_{\psi_n})^2\leq (\limsup_{n\to\infty} \int_{X_n}|\varphi_n|)\cdot (\limsup_{n\to\infty} \int_{X_n}|\psi_n|)\\  \leq (\limsup_{n\to\infty} D(\mathcal{F}_n)) \cdot (\limsup_{n\to\infty}  D(\mathcal{G}_n))\leq  D(\mathcal{F}) \cdot D(\mathcal{G})= (\int_{X}|\varphi |)\cdot (\int_{X}|\psi |)
\end{split}
\end{equation*}
which proves the desired inequality.

When $\Gamma$ is not of the first kind, we use \cite[Theorem 1.1]{Saric25} and the doubling procedure over the ideal boundary to reduce to the case of surfaces with the Fuchsian group of the first kind. Then, we apply the statement of the theorem to obtain the desired inequality on the double. The statement holds on the original surface by the symmetry with respect to the reflection. This method works for $X=\mathbb{D}$ as well by adding punctures along the unit circle after the doubling and noting that the intersection numbers converge (as in the proof of  \cite[Theorem 1.1]{Saric25}).
\end{proof}

The intersection of two arbitrary measured laminations on an infinite Riemann surface $X$ is not finite, in general. 
A corollary to the above theorem is that the intersection number of two measured geodesic laminations induced by the horizontal foliations of two finite-area holomorphic quadratic differentials is finite. This is not obvious since the leaves of the measured laminations can have intersection points going to infinity of $X$.

\begin{corollary}
\label{cor:finite-int-num}
Let $\varphi ,\psi\in A(X)$ and let $\mu_{\varphi}$, $\mu_{\psi}$ be measured laminations on $X$ which are homotopic to the horizontal foliations of $\varphi$ ,$\psi$. Then
$$
i(\mu_{\varphi},\mu_{\psi})<\infty .
$$
\end{corollary}

\vskip .2 cm

Next, we prove the continuity of the intersection number.

\begin{theorem}
\label{thm:int_numb_cont}
Let $X=\mathbb{H}/\Gamma$ be an arbitrary Riemann surface. Let $\varphi ,\varphi_n,\psi ,\psi_n\in A(X)$ and $\mathcal{F},\mathcal{F}_n,\mathcal{G},\mathcal{G}_n$ be the corresponding horizontal foliations. If $\int_X|\varphi -\varphi_n|\to 0$ and $\int_X|\psi -\psi_n|\to 0$ as $n\to\infty$, then
$$
\lim_{n\to\infty}i(\mu_{\mathcal{F}_n},\mu_{\mathcal{G}_n})=i(\mu_{\mathcal{F}},\mu_{\mathcal{G}}).
$$
\end{theorem}

\begin{proof}
We first prove the theorem when $X=\mathbb{D}$. Fix $\epsilon >0$. Then, there exists $r=r(\epsilon )<1$   such that 
$\mathbb{D}_r=\{ z:|z|\leq r\}$ satisfies
\begin{equation}
\label{eq:compact_int_eps}
\int_{\mathbb{D}\setminus \mathbb{D}_r}|\varphi |<\epsilon , \int_{\mathbb{D}\setminus \mathbb{D}_r}|\varphi_n |<\epsilon , \int_{\mathbb{D}\setminus \mathbb{D}_r}|\psi |<\epsilon ,\ \mathrm{and}\ \int_{\mathbb{D}\setminus \mathbb{D}_r}|\psi_n |<\epsilon .
\end{equation}

We prove that a horizontal trajectory that meets a compact subset of $\mathbb{D}$ cannot have its endpoints on $S^1$ too close to each other (see the following lemma). Given a horizontal leaf $h$ in $\mathbb{D}$, we denote by $G(h)$ the hyperbolic geodesic in $\mathbb{D}$ which shares the endpoints with $h$. For $0<r<1$, let $\mathbb{D}_r=\{ z:|z|<r\}$.

\begin{lemma}
\label{lem:compact_leaves}
Let $\varphi,\varphi_n\in A(\mathbb{D})$ be such that $\int_X|\varphi -\varphi_n|\to 0$ as $n\to\infty$. Given $0<r<1$, there exists $R=R(r)<1$ such that for every trajectory $h\in\mathcal{F}_{\varphi_n}$ or $h\in\mathcal{F_{\varphi}}$, if $G(h)\cap \mathbb{D}_R=\emptyset$ then $h\cap \mathbb{D}_r=\emptyset$. 
\end{lemma}

\begin{proof}
Assume the lemma is false. Then there exist a fixed $0<r<1$, a sequence $R_n\to 1$ and a sequence $\varphi_{k(n)}$ with horizontal trajectories $h_n$ such that 
$$
G(h_n)\cap \mathbb{D}_{R_n}=\emptyset
$$
while
$$
h_n\cap\mathbb{D}_r\neq\emptyset .
$$
By $\varphi_{n(k)}$ we mean either an infinite subsequence of $\varphi_n$ or a fixed element $\varphi_{n_0}$ or $\varphi$. We rename $\varphi_{k(n)}$ to $\varphi_n$ and consider all three cases simultaneously.

Let $r'<1$ be an arbitrary number that satisfies $r'>r$. A lower bound on the  distance in the $\|\cdot\|_{\varphi_n}$-metric between $\{ z:|z|=r\}$ and $\{ z:|z|=r'\}$ is positive. Indeed, if there are $z_{j(n)}$, $z_{j(n)}'$ with $|z_{j(n)}|=r$, $|z_{j(n)}'|=r'$ with arbitrary short $\|\varphi_{j(n)}\|$-distances, then common subsequences converge to $z_{\infty}$, $z_{\infty}'$ with $|z_{\infty}|=r$, $|z_{\infty}'|=r'$ that have zero $\|\cdot\|_{\varphi}$-distance. This is impossible. Therefore, we have a lower bound $d>0$ on all $\|\varphi_n\|$-distances between $|z|=r$ and $|z|=r'$. 

Since $G(h_n)\cap\mathbb{D}_{R_n}=\emptyset$, it follows that the endpoints of $G(h_n)$ on $S^1$ are arbitrarily close when $n$ is large enough. Let $a_n\neq b_n$ be the endpoints of $G(h_n)$, which are also the endpoints of $h_n$. By Strebel \cite[Lemma 1.1]{Strebel-geom-u-d}, given $0<\rho_1<1$ there exists a positive $\rho_2<\rho_1$ (independent of $n$) such that there is a semicircle centered at $a_n$ with a radius $\rho$, $\rho_2<\rho <\rho_1$,  whose $\|\varphi_n\|$-length is less than $\frac{d}{2}>0$ (where $d$ is a lower bound from the previous paragraph). The choice of $\rho_2$ is independent of the sequence $\varphi_n$ while the choice of $\rho$ depends on $\varphi_n$. First, we choose the number $\rho_1$ small enough so that the semicircle with center $a_n$ does not intersect $|z|=r'$. Then, for $n$ large enough, the point $b_n$ is inside the semicircle with radius $\rho_2$ and, therefore, inside the semicircle with radius $\rho$. Therefore, $h_n$ has a subarc with both endpoints on the semicircle with radius $\rho$ that contains $z_n\in \{z:|z|=r\}$. Since $\|\varphi_n\|$-geodesics are the shortest connections in $\mathbb{D}$, the length of this subarc is strictly less than $\frac{d}{2}$. On the other hand, the subarc connects $|z|=r$ and $|z|=r'$, so its length has to be at least $d$. This is a contradiction that proves the lemma.
\end{proof}

For the fixed $r<1$ which satisfies (\ref{eq:compact_int_eps}), let $R^*$ be the maximum of two values of $R$ given by the Lemma \ref{lem:compact_leaves} for $\varphi_n$ and $\psi_n$. We choose a finite cover by quadrilaterals $\{ Q_1,\ldots ,Q_k\}$ (used in the definition of the intersection number)
of the intersections of the supports of measured laminations $\mu_{\varphi}$ and $\mu_{\psi}$ in the closed disk $\mathbb{D}_{R^*}$. The set of the geodesics of the support of $\mu_{\varphi}$ that cross $\cup_{i=1}^kQ_i$ is compact. However, the corresponding horizontal trajectories may still accumulate on the unit circle. This happens when an arc of the unit circle $S^1$ does not contain endpoints of the horizontal trajectories (see Strebel \cite{Strebel-geom-u-d}).

Fix $Q_i$ and consider the set $\mathcal{H}_{Q_i}$ of horizontal trajectories of $\varphi$ that correspond to the support of $\mu_{\varphi}$ that crosses $Q_i$. The set $\mathcal{H}_{Q_i}$ has two ``bounding'' trajectories $h_i^{left}$ and $h_i^{right}$ such that all other trajectories are between $h_i^{left}$ and $h_i^{right}$. Since $\mathcal{H}_{Q_i}$ may accumulate to $S^1$, we include the possibility that either $h_i^{left}$ or $h_i^{right}$ are intervals on $S^1$. For a dense set of trajectories in $\mathcal{H}_{Q_i}$, we choose a non-singular vertical arc intersecting each and note that the whole set $\mathcal{H}_{Q_i}$ can be covered by the countable set of (horizontal) strips defined by these trajectories (see Strebel \cite[\S 13.5]{Strebel}). If $\mathcal{H}_{Q_i}$ does not accumulate to an arc of $S^1$, then it can be covered by finitely many such strips. Otherwise, we may have to use countably many strips covering $\mathcal{H}_{Q_i}$. If strips are $\{ S_j\}_j$, we make disjoint strips by taking complements as in $S_1'=S_1$, $S_2'=S_2\setminus S_1$, $S_3'=S_3\setminus (S_1\cup S_2)$, etc. (see Strebel \cite[page 87, \S 13.5]{Strebel}). In the case of the accumulation to the unit circle $S^1$, we note that the strips accumulating to $S^1$ are inside a set of small area as in (\ref{eq:compact_int_eps}). Therefore, we can omit those strips or substrips of such strips if they accumulate to $S^1$. Therefore, we have finitely many strips of horizontal trajectories covering $\mathcal{H}_{Q_i}$ except for a subset in (\ref{eq:compact_int_eps}), and the vertical arcs defining the strips are compactly contained in $\mathbb{D}$. 

The horizontal trajectories through the endpoints of the vertical arcs may not be totally regular (for the definition, see Strebel \cite{Strebel-geom-u-d}). If necessary, we slightly decrease each vertical trajectory such that the horizontal trajectories through endpoints are totally regular and the total measure of the trajectories excluded by decreasing the vertical trajectories is less than $\epsilon$ over all strips. We perform this process on each $\mathcal{H}_{Q_i}$ to obtain a covering by finitely many disjoint strips, possibly omitting a set with $\|\varphi \|$-area less than $k\epsilon$.  The horizontal leaves of $\varphi_n$ converge to the totally regular horizontal leaves of $\varphi$ in the Euclidean metric on the closure $\overline{\mathbb{D}}=\mathbb{D}\cup S^1$ (see Strebel \cite[Theorem 4.1]{Strebel-geom-u-d}). It follows that each strip of $\varphi$ can be approximated in the Euclidean metric by the strips of $\varphi_n$ by taking the vertical arcs of $\varphi_n$ that converge to the vertical arcs of the corresponding strip of $\varphi$. Therefore, the difference between the $\varphi$-strip and its approximation $\varphi_n$-strip has a small $\|\varphi \|$-area when $n$ is large, as well as the difference between the $\varphi_n$-strip that approximates $\varphi$-strip, has a small $\|\varphi_n \|$-area when $n$ is large. Additionally, we can arrange that the finitely many sets of the strips of $\varphi_n$ that approximate the finite family of strips of $\varphi$ chosen above are pairwise disjoint. 

We first consider the contributions to the intersection numbers of the set of horizontal trajectories of $\varphi_n$ that are outside the finite set of strips approximating the finite set of $\varphi$. 
We denote by $\mathcal{C}_n$ the partial measured foliation of $\mathbb{D}$ obtained by erasing the horizontal leaves of $\varphi_n$ that belong to the finitely many strips in the above approximation. Then $\mathcal{C}_n$ is a partial measure foliation defined by the imaginary parts of the natural parameter for $\varphi_n$. The partial measured foliation $\mathcal{C}_n$ is proper and has a finite Dirichlet integral. By the realization theorem \cite[Theorem 1.1]{Saric25}, there exists a finite-area holomorphic quadratic differential $\varphi_{\mathcal{C}_n}$ on $\mathbb{D}$ that realizes $\mathcal{C}_n$. The Dirichlet principle (Theorem \ref{thm:dir_principle}) gives that
$$
\int_{\mathbb{D}}|\varphi_{\mathcal{C}_n}|\leq D(\mathcal{C}_n).
$$
By Theorem \ref{thm:inters-extremal_l}, we have
$$
i(\mathcal{C}_n,\mu_{\psi_n})^2\leq (\int_{\mathbb{D}}|\varphi_{\mathcal{C}_n}|)\cdot (\int_{\mathbb{D}}|\psi_n|))\leq 
D(\mathcal{C}_n)\cdot  (\int_{\mathbb{D}}|\psi_n|))\leq const\cdot\epsilon 
$$
for $n$ large enough. The symmetric conclusion holds when we approximate the horizontal trajectories of $\psi$ corresponding to the finite quadrilaterals covering $\mathbb{D}_R$. Therefore, for both approximations, the finite parts contribute to the intersection $i(\mu_{\varphi_n},\mu_{\psi_n})$ with an error of at most a constant multiple of $\epsilon$.

Consider the contributions to the intersection number of the finitely many strips of $\varphi_n$ and $\psi_n$ approximating the finitely many strips of $\varphi$ and $\psi$. The intersection number $i(\mu_{\varphi},\mu_{\psi})$ is approximated by the intersections over the finitely many rectangles $\{ Q_1,\ldots ,Q_k\}$, which are divided into finitely many strips. Each hyperbolic geodesic homotopic to a horizontal trajectory of $\varphi$ crossing a rectangle $Q_i$ transversely intersects each hyperbolic geodesic homotopic to a horizontal trajectory of $\psi$ crossing the rectangle $Q_i$. Therefore, any two strips for $\varphi$ and $\psi$ corresponding to a single $Q_i$ transversely intersect, and the contribution to the intersection number is the product of their transverse measures. The transverse intersection is observed because the endpoints of each horizontal trajectory of $\varphi$ separate the endpoints of each horizontal trajectory of $\psi$ for the two strips. The convergence of the horizontal trajectories of the strips for $\varphi_n$ and $\psi_n$ in $\overline{\mathbb{D}}$ implies that the separation on the endpoints of the approximating strips is true for $n$ large enough (this is the reason that we use totally regular trajectories). This separation holds for all strips when $n$ is sufficiently large, as we consider only a finite number of strips. Moreover, the locally uniform convergence of $\varphi_n$ to $\varphi$ and $\psi_n$ to $\psi$ implies that the transverse measures of the approximating strips converge as well. Therefore 
$|i(\mu_{\varphi_n},\mu_{\psi_n})-i(\mu_{\varphi},\mu_{\psi})|$ is less than a constant multiple of $\epsilon$ for all $n$ large enough. Since $\epsilon >0$ is arbitrary, we conclude that the intersection numbers converge. This finishes the proof in the case $X=\mathbb{D}$.

\vskip .2 cm

We assume now that $X=\mathbb{D}/\Gamma$ with $\Gamma$ of the first kind. 
We first prove a lemma that controls the positions of the hyperbolic geodesics that are homotopic to horizontal leaves of holomorphic quadratic differentials. Unlike in the case of the unit disk $\mathbb{D}$ (see Lemma \ref{lem:compact_leaves}), the validity of the lemma is for an approximating sequence of a single finite-area holomorphic quadratic differential. We do not claim uniformity in the norm.

\begin{lemma}
\label{lem:compact_leaves_surface}
Let $X=\mathbb{D}/\Gamma$ with $\Gamma$ of the first kind.
Let $\varphi,\varphi_n\in A(X)$ be such that $\int_X|\varphi -\varphi_n|\to 0$ as $n\to\infty$. Given a finite area geodesic subsurface $C\subset X$ there exists a finite area geodesic subsurface $C_1\subset X$ such that for every trajectory $h\in\mathcal{F}_{\varphi_n}$ or $h\in\mathcal{F_{\varphi}}$, if $G(h)\cap C_1=\emptyset$ then $h\cap C=\emptyset$. 
\end{lemma}

\begin{proof}
Assume, on the contrary, that a sequence $h_n$ of distinct horizontal trajectories of $\varphi_n$ intersects a finite area geodesic subsurface $C$ of $X$ while the corresponding sequence of geodesics $G(h_n)$ leaves every finite area geodesic subsurface of $X$. The sequence $h_n$ has a subsequence that converges locally uniformly to a horizontal trajectory $h$ of $\varphi$, where $h$ is either regular or is composed of finitely or countably many singular trajectories that are concatenated. 
After renumbering, we denote the subsequences by $h_n$ and $\varphi_n$. 

Let $\{ X_k\}_k$ be an exhaustion of $X$ by nested finite area geodesic subsurfaces. 
Since $G(h_n)$ leaves every finite area geodesic subsurface of $X$, it follows that for each $k>0$ there exists $n_k\in \mathbb{N}$ such that
$$
G(h_n)\cap X_k=\emptyset 
$$
for all $n\geq n_k$.
In other words, the horizontal trajectories $h_n$ can be homotoped out of $X_k$ for all $n\geq n_k$. It follows that $h$ can be homotoped out of $X_k$ by choosing a large enough $n\geq n_k$ such that $h_n\cap X_k$ and $h\cap X_k$ are homotopic relative to $\partial X_k$. Since this is true for all $k$, the horizontal trajectory $h$ straightens to a hyperbolic geodesic $G(h)$ which does not intersect any simple closed geodesic on $X$. This is impossible since $\Gamma$ is of the first kind. The contradiction establishes the lemma.
\end{proof}

Fix $\epsilon >0$. Let $K_0\subset X$ be a finite area geodesic subsurface such that, for any $K'\supset K_0$,
$$
i(\mu_{\varphi},\mu_{\psi})<i_{K'}(\mu_{\varphi},\mu_{\psi})+\epsilon
$$
where $i_{K'}(\mu_{\varphi},\mu_{\psi})$ is the total intersection number in $K'$ of the measured geodesic laminations $\mu_{\varphi}$ and $\mu_{\psi}$.

Let $K$ be a finite area subsurface of $X$ with geodesic boundary such that $\int_{X\setminus K}|\varphi |<\epsilon$, $\int_{X\setminus K}|\varphi_n|<\epsilon$, $\int_{X\setminus K}|\psi_n |<\epsilon$ and $\int_{X\setminus K}|\psi_n|<\epsilon$ for all $n\in\mathbb{N}$. We erase all horizontal leaves of the foliations for $\varphi$, $\varphi_n$, $\psi$, and $\psi_n$ that intersect $K$. We are left with partial foliations of $X$ whose Dirichlet integrals are less than $\epsilon$. By Theorem \ref{thm:inters-extremal_l}, the intersection numbers between the obtained foliations are less than a constant multiple of $\epsilon$. 

For the subsurface $K$ defined above, there exists another finite area subsurface $K_1$ of $X$ such all horizontal trajectories of $\varphi$, $\varphi_n$, $\psi$, and $\psi_n$ that do not intersect $K$ are homotopic to the hyperbolic geodesics that do not intersect $K_1$ (this follows by Lemma \ref{lem:compact_leaves_surface}). 

Let $K_2$ be a finite area subsurface of $X$ that contains $K_0\cup K_1$. From the above discussion, we have 
$$
i(\mu_{\varphi},\mu_{\psi})<i_{K_2}(\mu_{\varphi},\mu_{\psi})+\epsilon 
$$
and by Theorem \ref{thm:inters-extremal_l}, for all $n\in\mathbb{N}$,
$$
i_{X\setminus K_2}(\mu_{\varphi_n},\mu_{\psi_n})<\epsilon .
$$
To finish the proof, we need to show that
\begin{equation}
\label{eq:conv-int-on-compact}
i_{K_2}(\mu_{\varphi_n},\mu_{\psi_n})\to i_{K_2}(\mu_{\varphi},\mu_{\psi})
\end{equation}
when $n\to\infty$. 

Unlike in the case of finite area holomorphic quadratic differentials on the unit disk, we do not have control over the endpoints of the lifts to $\mathbb{D}$ of the nearby horizontal trajectories on $X$. For this reason, the rest of the argument is indirect compared to the unit disk case.

Let $\omega$ be a Dirichlet fundamental polygon for the group $\Gamma$ on the unit disk $\mathbb{D}$. Since the geodesic laminations arising from finite area holomorphic quadratic differentials do not have leaves ending at the punctures, we can excise neighborhoods of the punctures with horocyclic boundaries of lengths $1$. The measured laminations do not enter these neighborhoods. Therefore, we can replace the finite area subsurface $K_2$ with a compact subsurface $K_2'$ (by excising the neighborhoods of the punctures) such that all geodesics (in the supports of the measured laminations we are considering) that intersect $K_2$ also intersect $K_2'$. Let $\tilde{K}_2'$ be the compact subset of $\omega$ that covers $K_2'$. 

Consider two simple geodesics $g_1$ and $g_2$ in $X$ intersecting $K_2'$. If $g_1\cap K_2$ and $g_2\cap K_2$ are homotopic in $K_2$ relative to the boundary $\partial K_2$, then $g_1$ and $g_2$ may still have an intersection point in $K_2'$. By knowing the intersection numbers between all simple closed curves in $K_2$ and the geodesic arcs $g_1\cap K_2$ and $g_2\cap K_2$, it is not possible to determine whether the geodesic arcs intersect in $K_2$. 
Let $\tilde{g}_1$ and $\tilde{g}_2$ be the lifts of $g_1$ and $g_2$ to $\mathbb{D}$ that intersect $\tilde{K}_2'$. If the geodesics $\tilde{g}_1$ and $\tilde{g}_2$ intersect inside $\tilde{K}_2'$ at a small angle at most $\alpha >0$, then their endpoints on $S^1$ are on a small Euclidean distance. Equivalently, two simple geodesics $g_1$ and $g_2$ intersect inside $K_2'$ at an angle at least $\alpha >0$ if there exists a finite area geodesic subsurface $K_3$ which contains $K_2$ such that $g_1\cap K_3$ and $g_2\cap K_3$ are not homotopic relative to the boundary $\partial K_3$. The surface $K_3$ depends on $K_2$ and $\alpha$.

Using the monotonicity of measures, it follows that the total intersection number between the geodesics in the supports of $\mu_{\varphi}$ and $\mu_{\psi}$ (inside $K_2$) that intersect at angles at most $\alpha >0$ is less than $\epsilon $ for $\alpha$ small enough. Because of this estimate, we restrict our attention to the pairs of geodesics (in the supports of $\mu_{\varphi}$ and $\mu_{\psi}$) that intersect $K_2$ and are non-homotopic in $K_3$ relative to $\partial K_3$. The measured geodesic laminations $\mu_{\varphi_n}$ and $\mu_{\psi_n}$ restricted to $K_3$ also define measured laminations on $K_3$ which have the same property that if two geodesic arcs intersect at an angle larger than $\alpha$ then the geodesic arcs are not homotopic modulo the boundary $\partial K_3$.

By Strebel (\cite[Theorem 24.7, page 162]{Strebel}, the intersection numbers $i(\gamma ,\mu_{\varphi_n})$ and $i(\gamma ,\mu_{\psi_n})$ converge to $i(\gamma ,\mu_{\varphi})$ and $i(\gamma ,\mu_{\psi})$ for all simple closed geodesics $\gamma$ on $K_3$. The number of homotopy classes of mutually disjoint simple arcs connecting $\partial K_3$ to itself is finite (depends only on the topology of $K_3$). Since the intersection numbers with simple closed geodesics converge, the transverse measures of the homotopy classes also converge. Therefore, the intersections $i_{K_2}(\mu_{\varphi_n},\mu_{\psi_n})$ are converging to $i_{K_2}(\mu_{\varphi},\mu_{\psi})$ since $i_{K_2}(\mu_{\varphi},\mu_{\psi})$ is the sum of the intersections of finitely many strips up to an error of at most $\epsilon$. If $\mu_{\varphi_n}$ (or $\mu_{\psi_n}$) has more homotopy classes of arcs connecting $\partial K_3$ to itself, then the measures are small again by the convergence of the intersection with simple closed curves.

If some subset of geodesics does not leave $K_3$, then they are divided into strips, and the intersection numbers again converge since the transverse measures $\mu_{\varphi}$ (or $\mu_{\psi}$) of the strips are approximated by $\mu_{\varphi_n}$ (or $\mu_{\psi_n}$). All the intersections outside $K_2$ are small, and therefore it is enough to consider the restrictions to $K_2$ of the intersections in $K_3$. By letting $\epsilon\to 0$ we conclude that 
$$
\lim_{n\to\infty}i(\mu_{\varphi_n},\mu_{\psi_n})=i(\mu_{\varphi},\nu_{\psi}). 
$$

\vskip .2 cm

It remains to consider the case when $X=\mathbb{D}/\Gamma$ with $\Gamma$ of the second kind. As in the proof of \cite[Theorem 1.1]{Saric25}, we choose finitely many points $A_k$ in each component of the open ideal boundary of $X$ and double $X$ across its open ideal boundary to obtain a Riemann surface $\widehat{X}$ whose covering group is of the second kind. Then $\widehat{X}_k=\widehat{X}\setminus A_k$ also has a covering group of the second kind, and $\cup_kA_k$ is dense in the open ideal boundary of $X$. The horizontal foliations of the holomorphic quadratic differentials $\varphi_n$, $\psi_n$, $\varphi$, and $\psi$ are measured foliations on $X$ whose leaves lift to the arcs in the universal covering of $X$ with exactly two endpoints (see \cite{MardenStrebel1}). Therefore, if a horizontal ray accumulates an open arc of the ideal boundary, then it has a unique endpoint on the open boundary. Thus, any ray accumulating on the open ideal boundary extends by doubling to a ray in $\widehat{X}$. We double the horizontal foliations, erase the leaves with endpoints in $A_k$, and also erase the leaves on $\widehat{X}_k$ that are homotopic to a point in $\widehat{X}_k$ or a single puncture (i.e., a point in $A_k$). By \cite[Theorem 1.1]{Saric25}, the partial foliations can be represented by the holomorphic quadratic differentials $\widehat{\varphi}_n^k$, $\widehat{\psi}_n^k$, $\widehat{\varphi}^k$, and $\widehat{\psi}^k$ on $\widehat{X}_k$, necessarily of finite area. The holomorphic quadratic differentials are invariant under the reflection across the ideal boundary of $X$. Therefore, the intersections with the cross-cuts on $X$ are equal to half of the intersections with the doubled closed curves on $\widehat{X}_k$. By Theorem \ref{thm:inters-extremal_l}, the contributions to the intersection numbers of the erased leaves are less than $\epsilon >0$ for all  $k\geq k_{\epsilon}$.
By the above-established result for covering groups of the first kind, we conclude that 
$$
|i(\mu_{\widehat{\varphi}_n^k}, \mu_{\widehat{\psi}_n^k})-i(\mu_{\widehat{\varphi}^k},\mu_{\widehat{\psi}^k})|<\epsilon
$$
for all $n\geq n_0$ and a fixed $k\geq k_{\epsilon}$. 

Denote by $\varphi_n^k$, $\psi_n^k$, $\varphi^k$, and $\psi^k$ the restrictions of $\widehat{\varphi}_n^k$, $\widehat{\psi}_n^k$, $\widehat{\varphi}^k$, and $\widehat{\psi}^k$ on $\widehat{X}_k$ to $X$, we conclude that
\begin{equation*}
\begin{split}
|i(\mu_{{\varphi}_n}, \mu_{{\psi}_n})-i(\mu_{{\varphi}},\mu_{{\psi}})|\leq |i(\mu_{{\varphi}_n}, 
\mu_{{\psi}_n})-i(\mu_{{\varphi}_n^k}, \mu_{{\psi}_n^k})|+
\\
|i(\mu_{{\varphi}_n^k}, \mu_{{\psi}_n^k})-i(\mu_{{\varphi}^k},\mu_{{\psi}^k})|+
|i(\mu_{{\varphi}^k}, \mu_{{\psi}^k})-i(\mu_{{\varphi}},\mu_{{\psi}})|\\
\leq 3\epsilon 
\end{split}
\end{equation*}
for all $n\geq n_0(\epsilon )$. Since $\epsilon >0$ is arbitrary, the convergence follows.
\end{proof}

\section{The approximation by Jenkins-Strebel differentials}

If $X$ is a parabolic Riemann surface, then every finite area holomorphic quadratic differential $\varphi$ on $X$ is the limit of a sequence of Jenkins-Strebel differentials $\varphi_n$ (see \cite[Theorem 1.2]{Saric24}). 
We show that this fact is false, in general, when $X$ is a non-parabolic Riemann surface.

We first give an example using the continuity of the intersection number.

\begin{example}
Let $X=[-\frac{1}{2},\frac{1}{2}]\times [-\frac{1}{2},\frac{1}{2}]\setminus A$, where $A=\{ \pm (\frac{1}{2}- \frac{1}{n}):n\geq 3\}$. Then $X$ is an infinite Riemann surface. We take $\varphi (z) dz^2 = -dz^2$. The horizontal foliation of $\varphi$ consists of vertical arcs in $X$. Let $\psi$ be the finite area holomorphic quadratic differential in $X$ which is represented by a partial measured lamination $\mathcal{H}$ on $X$ given by the horizontal arcs in $[-\frac{1}{2},\frac{1}{2}]\times [\frac{1}{4},\frac{1}{2}]$ with the Euclidean transverse measure (see Figure 3). 

\begin{figure}[h]
\begin{center}
{\centerline{\epsfig{file =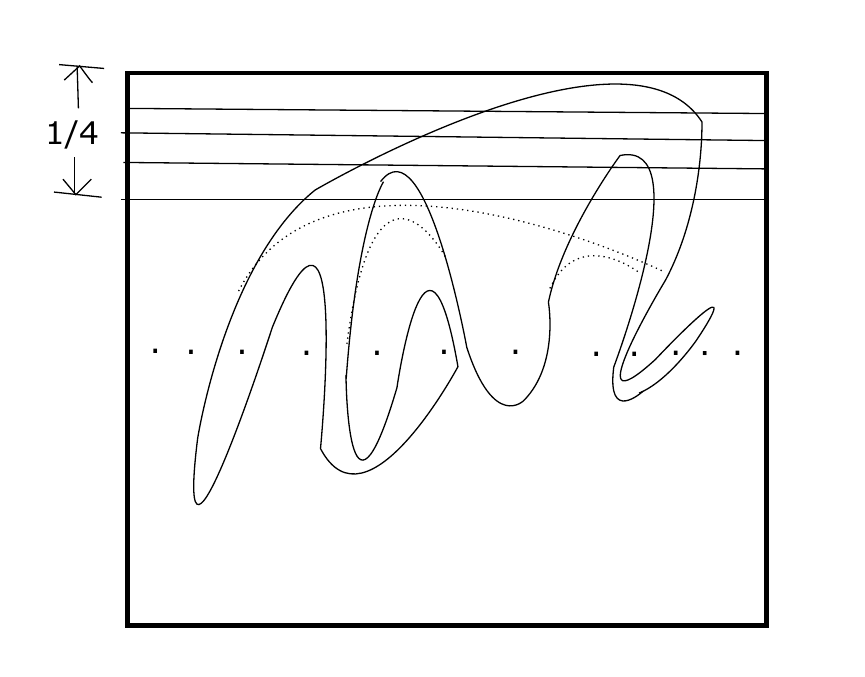,width=9.0cm,angle=0} }}
\vspace{-20pt}
\end{center}
\caption{The closed leaf homotopes off the horizontal leaves.}
\end{figure}

We assume that $\varphi$ is approximated in the $L^1$-norm by a sequence of Jenkins-Strebel differentials $\varphi_n$. 
Then we would have
$$
\lim_{n\to\infty} i(\mu_{\varphi_n},\mu_{\psi})=i(\mu_{\varphi},\mu_{\psi})=\frac{1}{4}.
$$
On the other hand, a simple closed non-trivial loop on $X$ can be homotoped such that it does not intersect any leaf of $\mathcal{H}$. Therefore, the left-hand side of the above equality is zero, which gives a contradiction. Thus $\varphi$ cannot be approximated by Jenkins-Strebel differentials. 
\end{example}

The above example easily generalizes to any Riemann surface $X=\mathbb{H}/\Gamma$, where $\Gamma$ is a group of the second kind.

\begin{theorem}
\label{thm:second-kind-non-approx}
Let $X=\mathbb{H}/\Gamma$ be a Riemann surface of the second kind. Then there exists a finite area holomorphic quadratic differential $\varphi$ on $X$ that cannot be approximated by a sequence of Jenkins-Strebel differentials in the $L^1$-norm.
\end{theorem}

\begin{proof}
Let $I$ be an open arc on the ideal boundary of $X$. Choose two Jordan domains $J_1$ and $J_2$ in $X$ such that the boundary of $J_1$ consists of two disjoint closed subarcs $I_1$ and $I_2$ of $I$, and two other Jordan arcs connecting the corresponding endpoints of $I_1$ and $I_2$. Similarly, the boundary of $J_2$ consists of two disjoint closed subarcs $I_3$ and $I_4$ of $I$ together with two more Jordan arcs connecting the corresponding endpoints of $I_3$ and $I_4$. We require that $I_1$, $I_2$, $I_3$ and $I_4$ are pairwise disjoint  and that $I_3$ separates $I_1$ and $I_2$. Finally, the boundaries of $J_1$ and $J_2$ intersect in four points (a minimal possible intersection, see Figure 4).

\begin{figure}[h]
\begin{center}
{\centerline{\epsfig{file =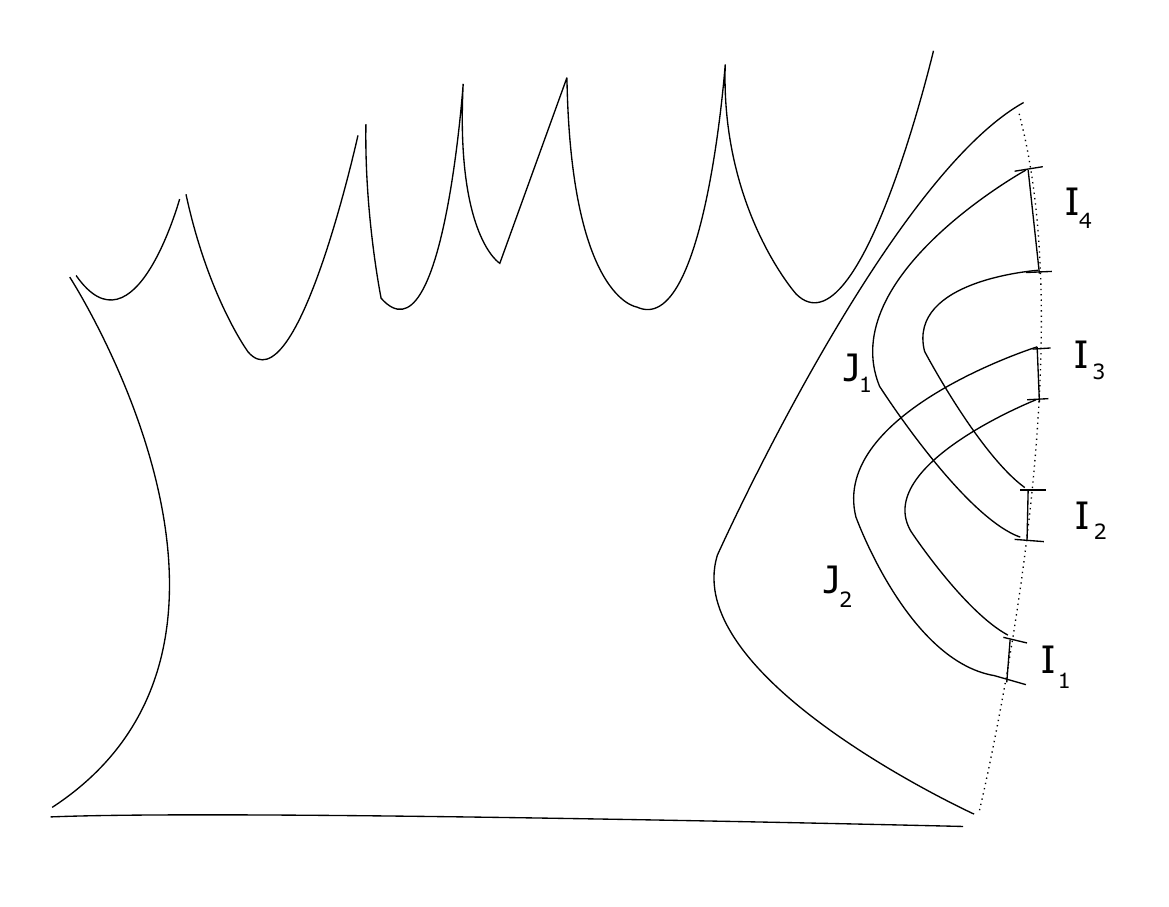,width=9.0cm,angle=0} }}
\vspace{-20pt}
\end{center}
\caption{The Jordan domains $J_1$ and $J_2$.}
\end{figure}

We map $J_1$ conformally to a Euclidean rectangle 
such that $I_1$ and $I_2$ are vertical sides. Then we denote by $\mathcal{F}_1$ the partial measured foliation on $X$ that is obtained by pulling back under the conformal map the horizontal foliation of $dz^2$ on the rectangle. Then the Dirichlet integral of $\mathcal{F}_1$ is finite and \cite[Theorem 1.1]{Saric25} implies that there exists a finite area holomorphic quadratic differential $\varphi$ whose measured horizontal foliation is homotopic to $\mathcal{F}_1$. Repeat the same process for $J_2$ and obtain a finite area holomorphic quadratic differential $\psi$ on $X$ which realizes $\mathcal{F}_2$. Let $v_i$ be the Euclidean height of the rectangle conformal to $J_i$ for $i=1,2$. Then we have
$$
i(\mu_{\varphi},\mu_{\psi})=v_1v_2>0.
$$
On the other hand, given a sequence of Jenkins-Strebel differentials $\varphi_n$ that approximates $\varphi$, each closed curve can be homotoped such that it does not intersect any leaf of $\mathcal{F}_2$. Therefore
$$
\lim_{n\to\infty}i(\mu_{\varphi_n},\mu_{\psi})=\lim_{n\to\infty}0=0,
$$
which is a contradiction.
\end{proof}

\end{document}